\documentclass[12pt]{article}
\usepackage{amsmath,amssymb,amsthm,fullpage}

\title{The Error in an Alternating Series}

\date{}

\theoremstyle{plain}
\newtheorem{thm}{Theorem}             

\theoremstyle{definition}

\makeatletter
\def\section{\@startsection{section}{1}{\z@}{-3.5ex plus -1ex minus
			  -.2ex}{2.3ex plus .2ex}{\large\bf}}
\def\subsection{\@startsection{subsection}{2}{\z@}{-3.25ex plus -1ex
			  minus -.2ex}{1.5ex plus .2ex}{\normalsize\bf}}
\makeatother

\renewcommand{\geq}{\geqslant}  
\renewcommand{\leq}{\leqslant}  

\newcommand{\Dl}{\Delta}        

\newcommand{\word}[1]{\quad\mbox{#1}\quad} 

\begin{document}

\maketitle

\section{Introduction} 

Mathematicians have studied the alternating series
\begin{equation}
\label{A} 
\sum_{n=1}^\infty (-1)^{n-1} a_n = a_1 - a_2 + a_3 - a_4 +\cdots
\end{equation}
since the dawn of analysis.  In January of 1714, the great
Gottfried Leibniz wrote a letter to
Johann Bernoulli
in which he explicitly stated his famous criterion for the convergence
of~\eqref{A} as well as the corresponding error estimate
\cite[p.~33]{Ferraro}.

\begin{thm} 
If $a_n \geq 0$ for $n = 0,1,\dots$ and if the sequence $(a_n)$ 
decreases monotonically to zero, then the series~\eqref{A} converges.
Let $L$ be its sum.  Moreover, let
\begin{align}
\label{Sn} 
S_n &:= a_1 - a_2 + a_3 - a_4 +\cdots+ (-1)^{n-1} a_n,
\\
R_n &:= L - S_n,
\nonumber
\end{align}
denote its $n^\mathit{th}$ partial sum and remainder, respectively.
Then
\begin{equation}
\label{E} 
|R_n| \leq a_{n+1},
\end{equation}
and $R_n$ has the sign~$(-1)^n$.
\end{thm}

It is amazing that the  error estimate~\eqref{E} remained virtually
unimproved for almost 250~years! Then, in 1962, 
Philip Calabrese, a sophomore (!) at the University of
Illinois, proved the following significant refinement~\cite{Cal}.

\begin{thm} 
Let $\Dl a_n := a_n - a_{n+1}$. If, additionally, the sequence
$(\Dl a_n)$ converges monotonically to zero, then
\begin{equation}
\label{E2} 
\frac{a_{n+1}}{2} < |R_n| < \frac{a_n}{2} \,.
\end{equation}
\end{thm}

Calabrese's refined error estimate allowed him to prove the very pretty result that the \emph{first} partial sum of Leibniz's series 
$$\ln 2=1-\frac{1}{2}+\frac{1}{3}-\frac{1}{4}+\cdots$$which approximates $\ln 2$ with four decimal places of accuracy is $S_{10000}$.

Seventeen years later, in~1979, Richard Johnsonbaugh
published the following refinement~\cite{John} of Calabrese's result.

\begin{thm} 
Let 
$\Dl^r a_n := \Dl^{r-1} a_n - \Dl^{r-1} a_{n+1}$ for
$r = 1,2,3,\dots,k$. If all the sequences $(\Dl^r a_n)$ for
$r=1,2,3,\dots,k$  decrease monotonically to zero, then 
\begin{equation}
\label{E3} 
\frac{a_{n+1}}{2} + \frac{\Dl a_{n+1}}{2^2}
+\cdots+ \frac{\Dl^{k} a_{n+1}}{2^{k+1}}
< |R_n| < \frac{a_n}{2} - \biggl\{
\frac{\Dl a_n}{2^2} +\cdots+ \frac{\Dl^k a_n}{2^{k+1}} \biggr\}.
\end{equation}
\end{thm}

For example, Johnsonbaugh uses  the sharper upper bound estimate \eqref{E3} with $k=2$ to prove the remarkably precise result that the \emph{first} partial sum, $S_n$, of the series
$$\frac{\pi}{4}=1-\frac{1}{3}+\frac{1}{5}-\frac{1}{7}+\cdots$$that approximates $\frac{\pi}{4}$ with four decimal places of accuracy is $S_{5000}.$  Here the precision means that although $4999$ terms of the series do not give four decimal places of accuracy, subtracting just one more term, $\dfrac{1}{9999}$, \emph{does} give it.  This underlines the slowness of the convergence of Leibniz's formula for $\pi$.  However, we will show below how to spectacularly accelerate the convergence of this same series.

The \emph{lower} bound for $|R_n|$ in Johnsonbaugh's theorem (and, in
particular, in Calabrese's) had already been found by
L.\,D. Ames in 1901 \cite{Ames}, but the upper bound is~new.

Finally, in 1985, Robert M. Young \cite{Young} used 
Cantor's theorem on nested intervals to give an elegant new
proof of Calabrese's refinement which makes the error estimates almost
intuitive. 

It does not seem to have been noticed that Young's method can be
adapted to give a new and transparent proof of Johnsonbaugh's
refinement.

We elaborate such a proof in this note.

\section{Proof of Johnsonbaugh's Theorem} 

Young points out that the crux of Leibniz's original proof of his
theorem is this: if the numbers $S_1,S_2,\dots$ satisfy the relation
\begin{equation}
\label{Young} 
S_{n+1} - S_n = (-1)^n a_{n+1},
\end{equation}
then, since the sequence $a_{n+1}$ decreases to zero, the sequence of
closed intervals
$$
[S_2,S_1], \quad [S_4,S_3], \quad [S_6,S_5],\ \cdots
$$
is~\emph{nested}. Therefore, Cantor's theorem on nested intervals
shows that there is a number~$L$ common to all these intervals, and
that
$$
L = \lim_{n\to\infty} S_n,
$$
whereby Leibniz's error estimate is immediate.

To prove Calabrese's theorem, Young defines
\begin{equation}
\label{Tn} 
T_n := S_n + (-1)^n \cdot \frac{a_{n+1}}{2} \,.
\end{equation}
Then
$$
T_{n+1} - T_n = (-1)^n \Dl a_{n+1},
$$
which has the same form as~\eqref{Young}. We note that 
$$
[S_{2r}, S_{2r-1}] \supseteq [T_{2r}, T_{2r-1}]
\word{for} r = 1,2,\dots
$$
Since $\Dl a_{n+1}$ decreases to zero, Cantor's theorem on nested
intervals shows that there is a number, $L_1 = L$, common to all the
intervals
$$
[T_2,T_1], \quad [T_4,T_3], \quad [T_6,T_5],\ \cdots
$$
and that 
$$
L=\lim_{n\to\infty} T_n.
$$
Now
$$
T_n - S_n
= (-1)^n \cdot \frac{a_{n+1}}{2}
$$
and
\begin{align*}
T_{n-1} - S_n
&= S_{n-1} + (-1)^{n-1} \cdot \frac{a_n}{2} - S_n
\\
&= (-1)^{n-1} \cdot \frac{a_n}{2} - (-1)^{n-1} a_n
= (-1)^n a_n.
\end{align*}
Therefore, if $n$ is \emph{even},
\begin{align*}
L - S_n &\geq T_n - S_n = \frac{a_{n+1}}{2} \,,
\\
L - S_n &\leq T_{n-1} - S_n = \frac{a_n}{2} \,,
\end{align*}
which are Calabrese's inequalities. A similar argument holds if $n$
is~\emph{odd}.

To prove Johnsonbaugh's theorem, we have to suitably
generalize~\eqref{Tn}. Let
\begin{align*}
T_n   &:= S_n + (-1)^n \cdot \frac{a_{n+1}}{2} \,,
\\
T'_n  &:= T_n + (-1)^n \cdot \frac{\Dl a_{n+1}}{2^2} \,,
\\
T''_n &:= T'_n + (-1)^n \cdot \frac{\Dl^2 a_{n+1}}{2^3} \,,
\\ 
\vdots &\hspace*{5em} \vdots
\\
T^{(k)}_n 
&:= T^{(k-1)}_n + (-1)^n \cdot \frac{\Dl^k a_{n+1}}{2^{k+1}} \,.
\end{align*}
Now we apply the reasoning we had already applied to $S_n$ and~$T_n$.
We note that 
$$
[S_{2r}, S_{2r-1}] \supseteq [T_{2r}, T_{2r-1}]
\supseteq [T'_{2r}, T'_{2r-1}]
\supseteq\cdots\supseteq [T^{(k)}_{2r}, T^{(k)}_{2r-1}]
$$
for $r = 1,2,\dots\,$.

Since $\Dl^k a_{n+1}$ decreases to zero, Cantor's theorem on nested
intervals shows that there is a number, $L_k = L$, common to all the
intervals
$$
[T^{(k)}_2, T^{(k)}_1], \quad [T^{(k)}_4, T^{(k)}_3], \quad
[T^{(k)}_6, T^{(k)}_5],\ \cdots
$$
and that 
$$
L=\lim_{n\to\infty} T^{(k)}_n.
$$
Substituting recursively in the definitions of $T^{(r)}_n$ for
$r = 1,2,\dots$, some simple algebra leads us to the equations
\begin{align*}
T^{(k)}_n - S_n
&= (-1)^n \biggr( \frac{a_{n+1}}{2} + \frac{\Dl a_{n+1}}{2^2}
+\cdots+ \frac{\Dl^k a_{n+1}}{2^{k+1}} \biggr),
\\
T^{(k)}_{n-1} - S_n
&= (-1)^n \biggr( \frac{a_{n}}{2} - \frac{\Dl a_{n}}{2^2}
-\cdots- \frac{\Dl^k a_{n}}{2^{k+1}} \biggr).
\end{align*}

Therefore, if $n$ is \emph{even},
\begin{align*}
L - S_n &\geq T^{(k)}_n - S_n 
= \frac{a_{n+1}}{2} + \frac{\Dl a_{n+1}}{2^2}
+\cdots+ \frac{\Dl^k a_{n+1}}{2^{k+1}} \,,
\\
L - S_n &\leq T^{(k)}_{n-1} - S_n 
= \frac{a_{n}}{2} - \frac{\Dl a_{n}}{2^2}
-\cdots- \frac{\Dl^k a_{n}}{2^{k+1}} \,,
\end{align*}
which are Johnsonbaugh's inequalities. A similar argument holds for
$n$~\emph{odd}. This completes the proof.
\qed  

\section{Euler's transformation}

We saw that Leibniz's series for $\log 2$ and $\dfrac{\pi}{4}$, while esthetically pleasing,  are useless for practical computation because of the slowness of convergence.  To rectify this situation, Euler stated the following important transformation formula:
\begin{thm}

If all the sequences $(\Dl^r a_n)$ for
$r=1,2,3,\dots$  decrease monotonically to zero, then the  ``Euler transform series" \begin{equation}\label{E}
\frac{a_1}{2}+\frac{\Dl a_1}{2^2}+\frac{\Dl^2 a_1}{2^3}+\frac{\Dl^3 a_1}{2^4}+\cdots
\end{equation}of the alternating series $a_1-a_2+a_3-+\cdots $ also converges and, indeed, to the same sum, $L$. If \begin{equation}
\label{En}
E_n: = \frac{a_1}{2}+\frac{\Dl a_1}{2^2}+\frac{\Dl^2 a_1}{2^3}+\frac{\Dl^3 a_1}{2^4}+\cdots+\frac{\Dl^{n-1} a_1}{2^{n}}
\end{equation}be the $n$-th partial sum, then the error, $r_n:=L-E_n$, in the approximation $L\approx E_n$ satisfies 
\begin{equation}
\label{rn}
0<r_n\leq \frac{\Dl^n a_1}{2^n}.
\end{equation}which, shows that $E_n$ underestimates $L$.
\end{thm}

\begin{proof}
A simple computation shows that that \begin{equation*}
\label{Tk}
T_{n+1}^{(k)}-T_{n}^{(k)}=(-1)^n\frac{\Dl^{k+1}a_{n+1}}{2^{k+1}}.
\end{equation*}By Cantor's theorem the intervals with end-points $T_{n}^{(k)}$ and $T_{n+1}^{(k)}$, in that or opposite order depending on the parity of $n$, close down on $L$, and in particular, $L$ is in all of them.   Therefore
\begin{equation}
\label{ }
|L-T_{n}^{(k)}|<\frac{\Dl^{k+1}a_{n+1}}{2^{k+1}}.
\end{equation}

Taking $n=0$ and then $k=n-1$, we obtain the inequality
\begin{equation}
\label{K}
\frac{a_1}{2}+\frac{\Dl a_1}{2^2}+\cdots+\frac{\Dl^{n-1} a_1}{2^{n}}<L<\frac{a_1}{2}+\frac{\Dl a_1}{2^2}+\cdots+\frac{\Dl^{n-1} a_1}{2^{n}}+\frac{\Dl^{n}a_{1}}{2^{n}},
\end{equation}But, the left-hand side is $E_n$, and therefore the error estimate \eqref{rn} is valid.  

\end{proof}

As an example of Euler's transformation let's compute $\dfrac{\pi}{4}$ with four decimal digit accuracy from Leibniz's series.  We already know that the first partial sum of the Leibniz series that achieves this accuracy is $S_{5000}$.  If we  compute the differences in the Euler transformation we find that 
$$\Dl^n a_1=2^n\frac{1\cdot 2\cdot 3\cdots n}{1\cdot 3\cdot 5\cdots 2n+1}.$$The error estimate shows that $n$ must satisfy
$$\frac{1\cdot 2\cdot 3\cdots n}{1\cdot 3\cdot 5\cdots 2n+1}\leq \frac{1}{20000}$$and the first value of $n$ which works is $n=13$ with an upper bound for the error given by $2.91*10^{-5}$.  Therefore, our theorem states that
\begin{align*}
E_{13}=&\frac{1}{2}\left\{1+\frac{1}{3}+\frac{1\cdot 2}{3\cdot 5}+\cdots +\frac{1\cdot 2\cdot 3\cdots 12}{1\cdot 3\cdot 5\cdots 25}\right\}\\ 
           =&\frac{1314078208}{1673196525}\\
           =&0.78536991...
\end{align*}approximates $\dfrac{\pi}{4}$ with an error no larger than $2.91*10^{-5}$.  In fact the true error is $2.852*10^{-5}$ which is impresively close to the predicted upper bound for the error.  Moreover \emph{one only needs $13$ summands} instead of $5000$ to reach the desired accuracy, an extraordinary acceleration of the rate of convergence.

In practice one normally computes a partial sum of the series exactly, and then uses the Euler transform to compute the remainder.  For example if we compute $S_{10}$ exactly and apply Euler's transform to the next $11$ summands we obtain $0.785398163$ which is correct to \emph{nine} decimal places!

We note that the inequality \eqref{K} appears as a problem (without a solution) on p.$270$ of \cite{Knopp}.

\subsection*{Acknowledgment}

We thank the referees for suggestions which substantially improved the style and content of this paper.  In particular we thank a referee for suggesting the relation of our method of proof to Euler's transformation. Support from the Vicerrector\'ia de Investigaci\'on of the 
University of Costa Rica is gratefully acknowledged.

\end{document}